\documentclass[a4paper,12pt]{amsart}

\makeindex

\usepackage{amsfonts,graphics,amsmath,amsthm,amsfonts,amscd,amssymb,amsmath,latexsym,euscript, enumerate}
\usepackage{epsfig,url}
\usepackage{flafter}
\usepackage[all,cmtip,line]{xy}
\usepackage{array}
\usepackage[english]{babel}
\usepackage{overpic}
\usepackage{subfig}
\usepackage{multirow}

\usepackage{wrapfig}
\usepackage{enumitem}

\usepackage[usenames,dvipsnames,svgnames,table]{xcolor}
\usepackage{graphicx}
\usepackage[bookmarks=true,bookmarksnumbered=true,
 pdftitle={},
 pdfsubject={},
 pdfauthor={},
 pdfkeywords={TeX, dvipdfmx, hyperref, color,},
 colorlinks=true, linkcolor=RoyalBlue, citecolor=Emerald]
 {hyperref}

\newtheorem{theorem}{Theorem}[section]
\newtheorem{lemma}[theorem]{Lemma}
\newtheorem{proposition}[theorem]{Proposition}
\newtheorem{conjecture}[theorem]{Conjecture}
\newtheorem{question}[theorem]{Question}

\theoremstyle{corollary}

\theoremstyle{definition} 

\newtheorem{definition-lemma}[theorem]{Definition-Lemma}

\newtheorem{example}[theorem]{Example}

\theoremstyle{remark}
\newtheorem{remark}[theorem]{Remark}

\numberwithin{equation}{section}

\newcommand{\C}{\mathbb{C}}

\newcommand{\Z}{\mathbb{Z}}

\newcommand{\Q}{\mathbb{Q}}

\def\P{\mathbb{P}}

\def\Pic{\operatorname{Pic}}

\def\mult{\operatorname{mult}}

\def\Bs{\operatorname{Bs}}

\def\Sing{\operatorname{Sing}}

\title[Pluri-fundamental divisors on Gorenstein Fano varieties of coindex $4$]
{Singularities of pluri-fundamental divisors on Gorenstein Fano varieties of coindex $4$}

\begin{document}

\author{Jinhyung Park}
\address{Department of Mathematical Sciences, KAIST, 291 Daehak-ro, Yuseong-gu, Daejeon 34141, Republic of Korea}
\email{parkjh13@kaist.ac.kr}

\thanks{J. Park was partially supported by the National Research Foundation (NRF) funded by the Korea government (MSIT) (NRF-2021R1C1C1005479).}

\date{\today}

\begin{abstract}
Let $X$ be a Gorenstein canonical Fano variety of coindex $4$ and dimension $n$ with $H$ fundamental divisor. Assume $h^0(X, H) \geq n -2$. 
We prove that a general element of the linear system $|mH|$ has at worst canonical singularities for any integer $m \geq 1$. When $X$ has terminal singularities and $n \geq 5$, we show that a general element of $|mH|$ has at worst terminal singularities for any integer $m \geq 1$.
When $n=4$, we give an example of Gorenstein terminal Fano fourfold $X$ such that a general element of $|H|$ does not have terminal singularities.
\end{abstract}

\maketitle

\section{Introduction}

Throughout the paper, we work over the field $\C$ of complex numbers.
Let $X$ be a Gorenstein Fano variety of dimension $n$ with canonical singularities. The \emph{index} of $X$ is
$$
i_X:=\max\{ t \in \Z \mid -K_X \sim tH \text{~where $H$ is an ample Cartier divisor} \}.
$$
It is well known that
$$
1 \leq i_X \leq n+1.
$$
The \emph{coindex} of $X$ is $n+ 1 - i_X$. An ample Cartier divisor $H$ on $X$ with $-K_X \sim i_X H$ is called the \emph{fundamental divisor} of $X$. Since $\Pic(X)$ is torsion-free, $H$ is uniquely determined up to linear equivalence. It is a natural problem to study singularities of general members in pluri-fundamental linear systems $|mH|$ for all integers $m \geq 1$.

By Kobayashi--Ochiai, $X$ is a projective space if $i_X=n+1$, and $X$ is a hyperquadric if $i_X=n$. 
A Gorenstein canonical Fano variety $X$ with $i_X=n -1$ is a \emph{del Pezzo variety}, and del Pezzo varieties were classified by Fujita \cite{Ft1, Ft3}. If $X$ is a del Pezzo variety, then the base locus $\Bs |H|$ is empty or consists of a single point neither in $\Sing X$ nor in $\Sing Y$, where $Y \in |H|$ is a general member. Thus $Y$ has canonical/terminal singularities if $X$ has canonical/terminal singularities. 
A Gorenstein canonical Fano variety $X$ with $i_X=n - 2$ is a \emph{Mukai variety}, and smooth Mukai varieties were classified by Mukai \cite{Mu} under the assumption that $|H|$ contains a smooth divisor. Mella \cite{M} verified this assumption, and moreover, he also proved that if $X$ is a Gorenstein Mukai variety with canonical/terminal singularities, then a general member in $|H|$ has canonical/terminal singularities except when $X$ is a complete intersection in $\P(1,1,1,1,2,3)$ of a quadric defined in the first four linear variables and a sextic. Finally, note that if $i_X \geq n -2$, then $|mH|$ is base point free for every integer $m \geq 2$ (see \cite[Remark 4.5]{L1}). We can conclude that singularities of general members in $|mH|$ with $m \geq 1$ are well understood when $i_X \geq n - 2$.

In this paper, we consider the case $i_X=n - 3$, i.e., $X$ has coindex $4$. 
Floris \cite{Fl} proved that a general member of the linear system $|H|$ has canonical singularities if $X$ is a Gorenstein canonical Fano variety of coindex $4$ and $h^0(X, H) \neq 0$.
However, in contrasts to the smaller coindex cases, there is a smooth Fano fourfold $X$ of coindex $4$ such that every member in $|H|$ is singular (see \cite[Example 2.12]{HV}). Heuberger \cite{H} proved that if $X$ is a smooth Fano fourfold, then a general member in $|-K_X|$ has only terminal singularities. This is a natural generalization of a classical result of Shokurov \cite{S} for smooth Fano threefolds. Heuberger's theorem together with aforementioned results implies that a general member in $|H|$ has only terminal singularities if $X$ is smooth.

The main result of this paper is the following.

\begin{theorem}\label{thm:main1}
Let $X$ be a Gorenstein canonical Fano variety of coindex $4$ and dimension $n \geq 4$, and $H$ be the fundamental divisor of $X$. 
\begin{enumerate}
 \item [$(1)$] Assume that $h^0(X, H) \geq n-2$. Then a general member of the linear system $|mH|$ has only canonical singularities for every integer $m \geq 1$.
 \item [$(2)$] Assume that $X$ has terminal singularities and $h^0(X, H) \geq n - 2$. Then a general member of the linear system $|mH|$ has only terminal singularities for every integer $m \geq 1$ unless $(n,m)=(4,1),(4,2),(4,3)$.
 \item [$(3)$] Assume that $X$ is smooth. Then $h^0(X, H) \geq n-2$, and a general member of the linear system $|mH|$ has only terminal singularities for every integer $m \geq 1$ unless $(n,m)=(4,2)$.
\end{enumerate}
\end{theorem}

If $X$ is a smooth Fano variety of coindex $4$ and dimension $n$, then Floris \cite[Theorem 1.2]{Fl} and Liu \cite[Theorem 1.2]{L1} showed that $h^0(X, H) \geq n - 2$. If $X$ is singular, then we do not know whether $H^0(X, H) \neq 0$. This nonvanishing follows from the following:
\begin{conjecture}[{Ambro--Kawamata effective nonvanishing conjecture \cite{A}, \cite{K}}]
Let $(X, \Delta)$ be a klt pair, and $D$ be a Cartier divisor on $X$. If $D$ is nef and $D-(K_X+\Delta)$ is nef and big, then $H^0(X, D) \neq 0$.
\end{conjecture}
\noindent This conjecture has been verified for low dimensional varieties \cite{K} and Fano weighted complete intersections \cite{PST}. Especially, \cite[Proposition 4.1 and Theorem 5.2]{K} say that if $X$ is a Gorenstein Fano fourfold with canonical singularities, then $h^0(X, H) \geq 2$. Although the methods of the present paper do not yield the results for higher coindex cases directly, we may still expect that Theorem \ref{thm:main1} for higher coindex would follow from the effective nonvanishing conjecture (cf. \cite{HS}).

In Theorem \ref{thm:main1} $(2)$, when $n=4$, one cannot expect that a general member in $|H|$ has terminal singularities. We give an example of Gorenstein terminal Fano fourfold $X$ such that a general member of the linear system $|H|$ does not have terminal singularities (see Example \ref{ex:ter->can} $(1)$).
In Theorem \ref{thm:main1} $(3)$, we do not know whether there is an example of a smooth Fano fourfold such that a general element in $|2H|$ does not have terminal singularities. See Remark \ref{rem:2H} for some partial result.

By \cite[Corollary 3]{L2}, if $X$ is a Gorenstein Fano variety of coindex $4$ and dimension $n$ with canonical singularities and $h^0(X, H) \geq n-2$, then $|mH|$ is base point free for any integer $m \geq 4$ (see Remark \ref{rem:ladder}). In particular, if $X$ is a smooth Fano variety of coindex $4$, then a general member in $|mH|$ is smooth for any integer $m \geq 4$.

One may expect that if a general member in $|H|$ has only mild singularities, then so does a general member in $|mH|$ for any $m \geq 2$. More generally, we may ask the following:

\begin{question}\label{ques:singlinsys}
Let $X$ be a smooth projective variety, and $L,M$ be divisors on $X$. Suppose that general members of $|L|$ and $|M|$ have only canonical/terminal singularities. Then does a general member of $|L+M|$ have also canonical/terminal singularities?
\end{question}

The answer is ``NO'' of course. Some counterexamples are given in Example \ref{ex:singlinsys}.

\subsection*{Organization} 
The paper is organized as follows. Section \ref{sec:main1} is devoted to proving Theorem \ref{thm:main1} $(1)$ and $(2)$. We also give some examples of terminal Fano fourfolds in which a general member of the fundamental linear system does not have terminal singularities (see Example \ref{ex:ter->can}). In Section \ref{sec:main2}, we negatively answer Question \ref{ques:singlinsys} in Example \ref{ex:singlinsys}, and we prove Theorem \ref{thm:main1} $(3)$.

\subsection*{Acknowledgments} 
The author would like to thank Miles Reid for calling attention to the paper \cite{CFHR} and In-Kyun Kim and Seung-Jo Jung for useful conversations. The authors is grateful to the referees for careful reading of the paper and valuable comments.

\section{Pluri-fundamental divisors on singular Fano varieties}\label{sec:main1}

In this section, we prove Theorem \ref{thm:main1} $(1)$ and $(2)$. For the definitions and basic properties of singularities of pairs, we refer to \cite{KM}.
We begin with fixing some notations.
Let $X$ be a Gorenstein Fano variety of coindex $4$ and dimension $n \geq 4$ with canonical singularities, and $H$ be the fundamental divisor on $X$. We have $-K_X = (n-3)H$. Assume that 
$h^0(X, H) \geq n - 2$.
Then $|mH| \neq \emptyset$ for each integer $m \geq 1$. Take a log resolution
$$
f_m \colon X_m \longrightarrow X
$$ 
of the ideal of the base locus $\Bs |mH|$. We may assume that $f_m$ is obtained by a sequence of blow-ups along smooth centers.
We write
$$
K_{X_m} = f_m^*K_X + \sum_{i} a_{m,i} E_{m,i}~~\text{ and }~~|f_m^*mH| = |M_m| + \sum_{i} r_{m,i} E_{m,i},
$$
where all $E_{m,i}$ are prime divisors, $|M_m|$ is the free part of $|f_m^*mH|$, and $\sum_{i} r_{m,i} E_{m,i}$ is the fixed part of $|f_m^*mH|$. By \cite[Theorem 1.1]{Fl}, a general member of $|H|$ is irreducible. 
Since $h^0(X, H) \geq 2$, it follows that $\dim \Bs|H| \leq n-2$; thus $\dim \Bs|mH| \leq n-2$ for each integer $m \geq 1$. Hence every $E_{m,i}$ is an $f_m$-exceptional divisor. Since $H$ is a Cartier divisor, all $a_{m,i}$ and $r_{m,i}$ are nonnegative integers.

\begin{lemma}\label{lem:dimbs}
$\dim \Bs|mH| \leq 2$ for any integer $m \geq 1$.
\end{lemma}

\begin{proof}
By \cite[Proposition 4.1]{Fl}, $(X,X_{n-1})$ is a plt($=$purely log terminal) pair, where $X_{n-1} \in |H|$ is a general member. As $X_{n-1}$ is connected, \cite[Proposition 5.51]{KM} shows that $X_{n-1}$ is irreducible and normal. By \cite[Theorem 5.50]{KM}, $X_{n-1}$ has Gorenstein canonical singularities. Note that $-K_{X_{n-1}} = ((n-1)-3)H_{n-1}$, where $H_{n-1}:=H|_{X_{n-1}}$. If $n \geq 5$, then $X_{n-1}$ is an $(n-1)$-dimensional Gorenstein canonical Fano variety of index $i_{X_{n-1}} \geq (n-1)-3$ with $h^0(X_{n-1}, H_{n-1}) \geq (n-1)-2$. If $i_{X_{n-1}} > (n-1)-3$, then $|H_{n-1}|$ is base point free (cf. \cite[Remark 4.5]{L1}) so that $(X_{n-1}, X_{n-2})$ is a plt pair, where $X_{n-2} \in |H_{n-1}|$ is a general member. If $i_{X_{n-1}}=(n-1)-3$, then by \cite[Proposition 4.1]{Fl}, $(X_{n-1}, X_{n-2})$ is also a plt pair. Continuing this process, we finally obtain a Calabi--Yau threefold $X_3$ with canonical singularities and $h^0(X_3, H_4|_{X_3}) \geq 1$. Notice that
$$
\Bs|H| = \Bs|H_{n-1}| = \cdots = \Bs|H_4|=\Bs|H_4|_{X_3}|.
$$
This shows $\dim \Bs|H| \leq 2$. Note that $\Bs|mH| \subseteq \Bs|H|$ for any $m \geq 2$. Then the lemma follows.
\end{proof}

\begin{remark}\label{rem:ladder}
$(1)$ If the Ambro--Kawamata effective nonvanishing conjecture is true for Gorenstein Fano variety of coindex $4$ with canonical singularities, then \cite[Proposition 4.1]{Fl} and the ``ladder'' argument as in the proof of Lemma \ref{lem:dimbs} show that $h^0(X, H) \geq n-2$.\\[3pt]
$(2)$ If $X$ is a Gorenstein Fano variety of coindex $4$ with canonical singularities and $h^0(X, H) \geq n-2$, then the ``ladder'' argument and \cite[Theorem 2]{L2} show that $|mH|$ is base point free for every integer $m \geq 4$.
\end{remark}

The following proposition, inspired by \cite[Proposition 9]{H}, is the key ingredient of the proof of Theorem \ref{thm:main1}.

\begin{proposition}\label{prop:key}
For an integer $m \geq 1$, we have
$$
a_{m,i} \geq \frac{m+n-3}{m} r_{m,i} - 1~~\text{ for all $i$}.
$$
\end{proposition}

\begin{proof}
Suppose that $a_{m,i} - \frac{m+n-3}{m} r_{m,i} < -1$ for some $i$.
Let
$$
c_0:= \inf\{c \mid \text{$a_{m,i} - cr_{m,i} \leq -1$ for some $i$} \}.
$$
Then $0 < c_0 < \frac{m+n-3}{m}$.
For an integer $k > c_0$, choose $k$ general members $D_1, \ldots, D_k \in |mH|$, and let $\Delta:=c_0 \cdot \frac{D_1+ \cdots + D_k}{k}$. Then the pair $(X, \Delta)$ is lc($=$log canonical) but not klt($=$Kawamata log terminal). Let $W$ be a minimal lc center of the lc pair $(X, \Delta)$.
Since $D_1, \ldots, D_k \in |mH|$ are general, $(X, \Delta)$ is klt outside the base locus $\Bs |mH|$ (cf. \cite[Lemma 5.1]{A}). Thus $W$ is contained in $\Bs |mH|$, so $\dim W \leq 2$ by Lemma \ref{lem:dimbs}.

By the generalization of Kawamata's subadjunction formula \cite[Theorem 1.2]{FG}, there exists an effective divisor $\Gamma$ on $W$ such that
$$
(K_X+\Delta)|_W \sim_{\Q} K_W + \Gamma
$$
and the pair $(W, \Gamma)$ is klt. Note that
$$
mH - (K_X + \Delta) \sim_{\Q} (m+n-3 - c_0 m)H.
$$
Since $c_0 m < m+n-3$, it follows that $mH - (K_X + \Delta)$ is ample. Then $mH|_W - (K_W + \Gamma)$ is ample. Recall that $\dim W \leq 2$. By \cite[Theorem 3.1]{K}, 
$H^0(W, mH|_W) \neq 0$.
Now, since $mH - (K_X + \Delta)$ is ample and $W$ is an lc center of the lc pair $(X, \Delta)$, we can apply \cite[Theorem 2.2]{F} to see that the restriction map
$$
H^0(X, mH) \longrightarrow H^0(W, mH|_W)
$$
is surjective. However, $W \subseteq \Bs |mH|$, so this restriction map is the zero map. We obtain $H^0(W, mH|_W)=0$, which is a contradiction. Thus the proposition holds.
\end{proof}

We are ready to prove Theorem \ref{thm:main1} $(1)$ and $(2)$.

\begin{proof}[Proof of Theorem \ref{thm:main1} (1) and (2)]
Recall that $X$ is a Gorenstein Fano variety of coindex $4$ and dimension $n$ with canonical singularities. We assume that $h^0(X, H) \geq n-2$. Let $Y_m \in |mH|$ be a general element for an integer $m \geq 1$.

\medskip 

\noindent $(1)$ We want to prove that $Y_m$ has canonical singularities. If $(X, Y_m)$ is a plt pair, then \cite[Theorem 5.50 and Proposition 5.51]{KM} imply that $Y_m$ has canonical singularities since $Y_m$ is Gorenstein. Thus it is enough to show that the pair $(X, Y_m)$ is plt. The birational morphism $f_m \colon X_m \to X$ is a log resolution of $(X, Y_m)$. We have
$$
K_{X_m} + f_{m,*}^{-1} Y_m = f_m^*(K_X + Y_m) + \sum_i (a_{m,i}-r_{m,i}) E_{m,i}.
$$
If $r_{m,i} = 0$, then 
$a_{m,i}-r_{m,i} \geq 0 > -1$.
If $r_{m,i} \geq 1$, then Proposition \ref{prop:key} implies that 
$$
a_{m,i} - r_{m,i} \geq \frac{n-3}{m}r_{m,i} -1 > -1.
$$ 
Thus $(X, Y_m)$ is a plt pair.

\medskip

\noindent $(2)$ Assume that $X$ has terminal singularities and $n=\dim X \geq 5$ or $m \geq 4$. We want to show that $Y_m$ has terminal singularities. If $m \geq 4$, then \cite[Theorem 2]{L2} (see also Remark \ref{rem:ladder} $(2)$) implies that $|mH|$ is base point free; hence $Y_m$ has terminal singularities. From now on, assume that $n \geq 5$ and $1 \leq m \leq 3$. We know that $Y_m$ is a normal projective variety with canonical singularities. Let $Y_m':=f_{m,*}^{-1} Y_m$ be the strict transform of $Y_m$ under $f_m$. Since $Y_m' \in |M_m|$ is a general element, $Y_m'$ is smooth. Then 
$$
f_m':=f_m|_{Y_m'} \colon Y_m' \longrightarrow Y_m
$$
is a log resolution of $Y_m$. We have
$$
K_{Y_m'} = {f_m'}^* K_{Y_m} + \sum_i (a_{m,i} - r_{m,i}) E_{m,i}|_{Y_m'}.
$$
Since $X$ has terminal singularities, we have $a_{m,i} \geq 1$ for all $i$.

Consider the case $m=1$. Note that $\frac{m+n-3}{m}=n-2 \geq 3$.
If $r_{1,i} \geq 1$, then Proposition \ref{prop:key} implies that 
$a_{1,i}-r_{1,i} \geq 2r_{1,i} - 1 >0$.
If $r_{1,i}=0$, then $a_{1,i} - r_{1,i} >0$. Thus $Y_1$ has terminal singularities.

Suppose now that $Y_m$ does not have terminal singularities for some $2 \leq m \leq 3$. Then there is some $i_0$ such that $a_{m,i_0} = r_{m,i_0} \geq 1$ and $E_{m,i_0}|_{Y_m'}$ is an $f_m|_{Y_m'}$-exceptional divisor. Since $Y_m$ has terminal singularities outside $\Bs|mH|$, we see that $f_m(E_{m,i_0}) \subseteq \Bs|mH|$.
Since $n \geq 5$ and $2 \leq m \leq 3$, we have $\frac{m+n-3}{m} \geq \frac{5}{3}$. If $r_{m,i} \geq 2$, then Proposition \ref{prop:key} implies that
$a_{m,i}-r_{m,i} \geq \frac{2}{3}r_{m,i} - 1 > 0$. Thus $a_{m,i_0}=r_{m,i_0} = 1$. If $f_m(E_{m,i_0}) \not\subseteq \Sing X$, then $\dim f_m(E_{m,i_0}) = n-2$ since $f_m$ is a composition of smooth center blow-ups. This means that $f_m(E_{m,i_0})$ is a divisor on $Y_m$ and $E_{m,i_0}|_{Y_m'}$ is not an $f_m|_{Y_m'}$-exceptional divisor. Thus $f_m(E_{m,i_0}) \subseteq \Sing X$, and $\dim f_m(E_{m,i_0}) \leq n-3$ because $X$ has terminal singularities.
By taking further blow-ups, we may assume that $f_1=f_m$ and $X_1=X_m$. Then there is an $i_1$ such that $E_{1,i_1} = E_{m,i_0}$. We have $a_{1,i_1}=a_{m,i_0} = 1$. Now, since $f_m(E_{m, i_0}) \subseteq \Bs|mH| \subseteq \Bs|H|$, it follows that $r_{1,i_1} \geq 1$. Thus $a_{1,i_1} - r_{1,i_1} \leq 0$. Note that $E_{1,i_1}|_{Y_1'}$ is an $f_1|_{Y_1'}$-exceptional divisor. We get a contradiction to that $Y_1$ has terminal singularities. Hence $Y_m$ has terminal singularities for any $2 \leq m \leq 3$. 
\end{proof}

Finally, we provide some examples of terminal Fano fourfolds in which a general element in the fundamental linear system does not have terminal singularities.

\begin{example}\label{ex:ter->can}
$(1)$ Let $Z:=X_{2,6}$ be a complete intersection in $\P(1,1,1,1,2,3)$ of a general quadric defined in the first four linear variables $x_0, \ldots, x_3$ and a general sextic (cf. \cite[Theorem 1]{M}). Then $Z$ is a Gorenstein terminal Fano threefold of index $1$, and $\Sing Z = \{p=(0:0:0:0:-1:1)\}$. A general member of $|H_Z|$ is singular at $p$, where $H_Z=-K_Z$ is the fundamental divisor of $Z$. Let $X:=Z \times \P^1$ so that $X$ is a Gorenstein Fano fourfold of index $1$ with terminal singularities. Note that $H=-K_X = \pi_1^*(-K_Z) + \pi_2^*(-K_{\P^1})$ is the fundamental divisor of $X$, where $\pi_1 \colon X \to Z$ and $\pi_2 \colon X \to \P^1$ are projections. A general element $Y$ in $|H|$ has one dimensional singular locus $\{p\} \times \P^1$. Since $\dim Y = 3$, it follows that $Y$ does not have terminal singularities. Here $Y$ is a Gorenstein Calabi--Yau threefold with canonical singularities. \\[3pt]
$(2)$ Let $X:=X_9$ be a weighted hypersurface in $\P(1,1,1,1,3,3)$ of degree $9$ (cf. quasismooth Fano 4-fold hypersurfaces ID 8 in \cite{GRDB} based on \cite{BK}).
Then $X$ is a non-Gorenstein $\Q$-Fano fourfold with terminal singularities such that $-K_X$ is a hyperplane with $(-K_X)^4=1$. Note that $\Sing X$ consists of three terminal singular points of the type $\frac{1}{3}(1,1,1,1)$. A general element in $|-K_X|$ is a weighted hypersurface $S_9$ in $\P(1,1,1,3,3)$ of degree $9$, and $S_9$ is a Gorenstein canonical Calabi--Yau threefold. Note that $\Sing S_9$ consists of three (non-terminal) canonical singular points of the type $\frac{1}{3}(1,1,1)$.
\end{example}

\section{Pluri-fundamental divisors on smooth Fano varieties}\label{sec:main2}

In this section, we first answer Question \ref{ques:singlinsys} by constructing smooth projective varieties $X$ and divisors $M$ such that general members in $|M|$ are smooth but all members in $|mM|$ are not normal for some $m \geq 2$, and then prove Theorem \ref{thm:main1} $(3)$.

\begin{example}\label{ex:singlinsys}
$(1)$ If $E$ is an exceptional divisor on a smooth projective variety, then $|mE|=\{mE\}$ for all $m \geq 1$. Now, $E$ is smooth, but $mE$ is non-reduced for any $m \geq 2$.\\[3pt]
$(2)$ Let $C$ be a smooth projective curve of genus $2$. There are two distinct points $P, Q$ on $C$ such that $2P \sim 2Q \sim K_C$. In particular, $Q-P \in \Pic^0(C)$ is a $2$-torsion. 
We can also find $\tau \in \Pic^0(C)$ such that $H^0(C, P+\tau)=H^0(C, Q-P+\tau)=H^0(C, 2\tau)=0$.
Let $E:=\mathcal{O}_C(P) \oplus \mathcal{O}_C(\tau)$, and $S:=\P(E)$ with the natural projection $\pi \colon S \to C$ and the tautological divisor $H$, i.e., $\mathcal{O}_S(H)=\mathcal{O}_{\P(E)}(1)$. Let $A:= H$ and $B:=H + \pi^*(Q-P)$. Then $A,B$ are sections of $\pi$, so they are smooth irreducible curves isomorphic to $C$. Furthermore, $A,B$ satisfy the following:
\begin{enumerate}
 \item[$\bullet$] $A^2=B^2=A.B=1$,
  \item[$\bullet$] $A \not \sim B$ but $2A \sim 2B$,
 \item[$\bullet$] $h^0(S, A)=h^0(S, B)=1$, and
 \item[$\bullet$] $h^0(S, 2A)=h^0(S, 2B)=2$.
\end{enumerate} 
Notice that $A,B$ meet at one point $p$ on $S$ and every member of $|2A|=|2B|$ has multiplicity at least $2$ at $p$. Thus every member in $|2A|=|2B|$ is not normal. \\[3pt]
$(3)$ \cite[Example 5.9]{PST}
For an integer $m \geq 1$, let 
$$
X=X_{(2m+1)(2m+2)} \subseteq \P(\underbrace{1, \ldots, 1}_{1+2m(2m+1)}, 2m+1, 2m+2)
$$
 be a weighted hypersurface of degree $(2m+1)(2m+2)$. Then $X$ is a smooth Fano variety of index $2$. If $H$ is the fundamental divisor of $X$, then a general member of $|H|$ is smooth. However, $|-2iH|$ does not contain a smooth member for any $1 \leq i \leq m$. In this case, a general member in $|-2iH|$ has terminal singularities.
\end{example}

We now turn to the proof of Theorem \ref{thm:main1} $(3)$. 

\begin{proof}[Proof of Theorem \ref{thm:main1} (3) except the case $(n,m)=(4,3)$]
Let $X$ be a smooth Fano variety of coindex $4$ and dimension $n$ with fundamental divisor $H$. By Theorem \ref{thm:main1} $(2)$, we only have to consider the cases $(n,m)=(4,1), (4,3)$. If $(n,m)=(4,1)$, then $H=-K_X$. Now, \cite[Theorem 2]{H} says that a general element in $|-K_X|$ has terminal singularities.
\end{proof}

\begin{remark}
Let $X$ be a smooth Fano variety of coindex $4$, and $H$ be the fundamental divisor of $X$. By \cite[Theorem 4]{L2}, $|mH|$ is base point free for any integer $m \geq 4$; hence a general element $Y_m \in |mH|$ is smooth in this case. But there is a smooth Fano fourfold $X$ of coindex $4$ such that every member in $|H|$ is singular (see \cite[Example 2.12]{HV}).
\end{remark}

To finish the proof of Theorem \ref{thm:main1}, it only remains to prove that if $H$ is the fundamental divisor of a smooth Fano fourfold $X$ of coindex $4$, then a general element $Y \in |3H|$ has terminal singularities. We know that $Y$ has canonical singularities.
 As in Section \ref{sec:main1}, take a log resolution $f \colon X_3 \to X$ of the ideal of the base locus $\Bs|3H|$. We may assume that $f$ is isomorphic outside $\Bs|3H|$ and it is obtained by a sequence of blow-ups along smooth centers. We write
$$
K_{X_3} = f^*K_X + \sum_{i} a_i E_i~~\text{ and }~~|f^*3H|=|M| + \sum_i r_i E_i,
$$
where all $E_i$ are $f$-exceptional prime divisors and $|M|$ is the free part of $|f^*3H|$ and $\sum_i r_i E_i$ is the fixed part of $|f^*3H|$. We may assume that $f(E_i) \subseteq \Bs|3H|$ for all $i$. All $a_i$ and $r_i$ are positive integers.

\begin{lemma}\label{lem:mult=2}
If a general element $Y$ in $|3H|$ has at worst isolated singularity at $x$ and $\mult_x Y \leq 2$, then $Y$ has terminal singularity at $x$.
\end{lemma}

\begin{proof}
We may assume that $f$ factors through the blow-up of $X$ at $x$ with exceptional divisor $E_{i_0}$. We have $a_{i_0} = 3$ and $r_{i_0} \leq 2$, so $a_{i_0} - r_{i_0} > 0$.
For every $f$-exceptional divisor $E_i$ with $f(E_i) = \{x\}$ but $E_i \neq E_{i_0}$, we have $a_i \geq 4$ since $f$ is a composition of smooth center blow-ups.
It is impossible that $a_i = r_i  \geq 4$ because Proposition \ref{prop:key} says that $a_i \geq \frac{4}{3}r_i - 1 > r_i$ when $r_i \geq 4$. Thus $a_i - r_i > 0$, and hence, $Y$ has terminal singularity at $x$.
\end{proof}

\begin{lemma}\label{lem:dimbs|mH|}
$\dim\Bs |mH| \leq 1$ for any integer $m \geq 2$. In particular, $\dim \Sing Y \leq 1$.
\end{lemma}

\begin{proof}
Suppose that $\dim \Bs |mH| \geq 2$ for some integer $m \geq 2$. By Lemma \ref{lem:dimbs}, we have $\dim \Bs |mH|=2$, so there is an irreducible surface $S \subseteq \Bs|mH| \subseteq \Bs|H|$. Now, take two general elements $D_1, D_2 \in |H|$. By Proposition \ref{prop:key}, $(X, D_1+D_2)$ is an lc pair, and $S$ is an lc center of $(X, D_1+D_2)$. There is a minimal lc center $C$ of $(X, D_1+D_2)$ contained in $S$. By \cite[Theorem 1.2]{FG}, there is an effective divisor $\Gamma$ on $C$ such that 
$$
(K_X+D_1+D_2)|_C \sim_{\Q} K_C + \Gamma
$$ 
and $(C, \Gamma)$ is a klt pair. By \cite[Theorem 3.1]{K}, $H^0(C, mH|_C) \neq 0$ since $mH - (K_X + D_1+D_2) \sim (m-1)H$ is ample. Now, by \cite[Theorem 2.2]{F}, the restriction map
$$
H^0(X, mH) \longrightarrow H^0(S, mH|_S)
$$
is surjective. However, $S \subseteq \Bs|mH|$, so this restriction map is the zero map. We get a contradiction. Therefore, $\dim\Bs |mH| \leq 1$ for any integer $m \geq 2$. Now, since $\Sing Y \subseteq \Bs|3H|$, it follows that $\dim \Sing Y \leq 1$.
\end{proof}

\begin{proof}[Proof of Theorem \ref{thm:main1} (3) for the case $(n,m)=(4,3)$]
We want to prove that a general element $Y \in |3H|$ has terminal singularities. Note that $H=-K_X$ and $\dim\Bs|3H| \leq 1$ by Lemma \ref{lem:dimbs|mH|}. We know that $h^0(X, H) \geq 2$.

First, assume that $H^4 \geq 2$. Take a general element $Z \in |H|$, which is a Gorenstein Calabi--Yau threefold with terminal singularities. 
Suppose that $\dim \Bs|3H|=1$. Then $Z$ is nonsingular at a general point $x$ in $\Bs|3H|$. By \cite[Theorem 3.1]{K1}, $|3H|_Z|$ is base point free at $x$. But $x \in \Bs|3H|=\Bs|3H|_Z|$, so we get a contradiction. Thus $\dim \Bs|3H| \leq 0$. 
Suppose that $Y$ has non-terminal singularity at $x$.
By Lemma \ref{lem:mult=2}, $\mult_x Y \geq 3$. Now, by Proposition \ref{prop:key}, $(X, Z+Y)$ is an lc pair. Thus $\mult_x Z = 1$, so $Z$ is nonsingular at $x$. By \cite[Theorem 3.1]{K1}, $3H|_Z$ is base point free at $x$, so we get a contradiction as before. Hence $Y$ has at worst terminal singularities.

Next, assume that $H^4 = 1$. The sectional genus of the polarized pair $(X, H)$ is $g(X, H)=\frac{(K_X + 3 H).H^3}{2}+1 = 2$.
By Fujita's classification \cite[Proposition C]{Ft}, we have $2 \leq h^0(X, H) \leq 4$, and the following hold:\\[3pt]
\begin{tiny}$\bullet$\end{tiny} $h^0(X, H)=4$ $\Leftrightarrow$ $X=X_{10} \subseteq \P(1,1,1,1,2,5)$ is a hypersurface of degree 10.\\
\begin{tiny}$\bullet$\end{tiny} $h^0(X, H)=3$ $\Leftrightarrow$ $X=X_{6,6} \subseteq \P(1,1,1,2,2,3,3)$ is a complete intersection of type (6,6).\\[3pt]
If $h^0(X, H)=4$, then $\Bs|H|=\Bs|3H|=\{ x\}$ and $|2H|$ is base point free. In this case, $\mult_x Y = 1$, so $Y$ is smooth. If $h^0(X, H)=3$, then $|3H|$ is base point free so that $Y$ is smooth. We now suppose that $h^0(X, H) = 2$.\footnote{\begin{scriptsize} It is unknown whether there is a smooth Fano fourfold $X$ with $h^0(X, -K_X)=2$ (cf. \cite[Question 5]{L2}).\end{scriptsize}} By Riemann--Roch formula, we have
$$
h^0(X, mH) = \frac{m^2(m+1)^2}{24} H^4+ \frac{m(m+1)}{24} H^2. c_2(X) + 1.
$$
Then $H^2. c_2(X) = 10$, and $h^0(X, 2H)=5,~h^0(X, 3H)=12$.
Let $Z_1, Z_2 \in |H|$ and $W \in |2H|$ be general members. Then $S:=Z_1 \cap Z_2$ is an irreducible Gorenstein surface with $K_S=H|_S$, and $C:=Z_1 \cap Z_2 \cap W$ is a Gorenstein curve with $K_C=3H|_C$. We have $H^i(X, \ell H)=0$ for $1 \leq i \leq 3$ and $\ell \in \Z$, so we get
$$
\begin{array}{l}
h^0(Z_1, H|_{Z_1})=1,~h^0(Z_1, 2H|_{Z_1})=3,~h^0(Z_1, 3H|_{Z_1})=7\\
h^0(S, H|_S)=0,~h^0(S, 2H|_S)=2,~h^0(S, 3H|_S)=4\\
h^0(C, H|_C)=0,~h^0(C, 2H|_C)=1,~h^0(C, 3H|_C)=4.
\end{array}
$$
Thus $p_a(C)=h^0(C, 3H|_C)=4$. As $C.H|_S=2$, we see that $C$ has at most two irreducible components. If $C$ is non-reduced, then $C=2H'$ for some $H' \in |H|_S|$. However, since $h^0(S, H|_S)= 0$, it follows that $C$ is reduced.

Suppose that there is an irreducible curve $A$ on $X$ with $A \subseteq \Bs|2H| \cap \Bs|3H|$. Since $h^0(S, C)=h^0(S, 2H|_S)= 2$, it follows that $C$ has two irreducible components. We write $C=A+B$ on $S$. Since the restriction map $H^0(X, 3H) \to H^0(C, K_C)$ is surjective, we have $A \subseteq \Bs|K_C|$.
Note that
$$
\deg_A(K_C)=\deg_B(K_C)=3
$$
since 
$$
\deg_A(H|_C)=A.H=1~~\text{ and }~~ \deg_B(H|_C)=B.H=1.
$$
By \cite[Definition 2.1 and Formula (3)]{FT}, we have
$$
4=p_a(C)=p_a(A)+p_a(B)+A \cdot B - 1,
$$
where 
$$
A \cdot B := \deg_A(K_C) - 2p_a(A)+2=\deg_B(K_C)-2p_a(B)+2.
$$
If $A \cdot B \geq 2$, then $C$ is numerically 2-connected in the sense of \cite[Definition 3.1]{CFHR}. In this case, by \cite[Theorem 3.3]{CFHR}, $|K_C|$ is base point free, so we get  a contradiction to that $A \subseteq \Bs|K_C|$. Thus $A \cdot B=1$, and then, $p_a(A)=p_a(B)=2$. Consider an exact sequence
$$
0 \longrightarrow \omega_B \longrightarrow \omega_C \longrightarrow \omega_C|_A \longrightarrow 0,
$$
which induces the following exact sequence
$$
0 \longrightarrow H^0(B, K_B) \longrightarrow H^0(C, K_C) \longrightarrow H^0(A, K_C|_A) 
$$
Then $h^0(A, K_C|_A) \geq 2$, which is a contradiction to that $A \subseteq \Bs|K_C|$. Thus we obtain $\dim \Bs|2H| \cap \Bs|3H| \leq 0$.

Recall that $Z_1$ is a Gorenstein Calabi--Yau threefold with terminal singularities. Then $\dim \Sing Z_1 \leq 0$. If $Y$ is singular along a curve $D$, then $D \subseteq \Bs|3H|$ and $D \not\subseteq \Bs|2H|$. For a general point $x \in D$, we have $\mult_x |H| =1$ and $\mult_x |2H|=0$, so $\mult_x Y = 1$ by the upper semicontinuity of the multiplicity. We get a contradiction because $Y$ is singular at $x$. This means that  $Y$ cannot be singular along a curve. By Bertini's theorem, we see that $\mult_x Y \leq 2$ for all $x \not\in \Bs|2H| \cup \Sing Z_1$.
Now, suppose that $Y$ has an isolated non-terminal singular point $x$. By Lemma \ref{lem:mult=2}, $\mult_x Y \geq 3$, which implies that $x \in \Bs|2H| \cup \Sing Z_1$. Note that every general element $Y' \in |3H|$ has $\mult_x Y' \geq 3$.
By Proposition \ref{prop:key}, $(X, Y+Z_1)$ is an lc pair, so $\mult_x Z_1 = 1$. If $\mult_x W  = 1$, then the upper semicontinuity of the multiplicity shows that $\mult_x Y' \leq 2$, which is a contradiction. Thus $\mult_x W \geq 2$. Now, $\dim \Bs|2H| \cap \Bs|3H| \leq 0$ implies that $C \cap Y'$ has dimension zero. Notice that $C \cap Y' = Z_1 \cap Z_2 \cap W \cap Y'$ has length 6, and recall that $\mult_x W \geq 2$ and $\mult_x Y' \geq 3$. Hence $C \cap Y'$ is indeed supported at a single point $x$.
Since $H^0(X, 3H) \to H^0(C, 3H|_C)$ is surjective, every element in $|3H|_C|$ has a single support $x$. But this is impossible since $h^0(C, 3H|_C) \geq 2$. We can conclude that  $Y$ has at worst terminal singularities.
\end{proof}

\begin{remark}\label{rem:2H}
Let $X$ be a Fano fourfold of coindex $4$ with fundamental divisor $H=-K_X$.
Suppose that $H^4 \geq 4$,~ $H^2.S \geq 3$ for every irreducible surface $S$, and $H^3.C \geq 2$ for every irreducible curve $C$. Take a general element $Z \in |H|$. By \cite[Theorem 3.1]{K}, $|2H|_Z|$ is base point free at every nonsingular point in $Z$. This implies that $\dim \Bs|2H| \leq 0$. In this case, we can easily show that a general member in $|2H|$ has terminal singularities.
\end{remark}


\end{document}